\documentclass[reqno,10pt]{amsart}
\pdfoutput=1

\usepackage[utf8]{inputenc}
\usepackage[english]{babel}
\usepackage{amsfonts}
\usepackage{amssymb}
\usepackage{bbm}
\usepackage{tikz}
\usepackage{svg}
\usepackage{scalerel}
\usepackage[right=2.5cm,bottom=2.5cm,footskip=30pt,left=2.5cm]{geometry}
\usepackage[autostyle]{csquotes}
\MakeOuterQuote{"}
\usepackage{color}
\usepackage{mathrsfs}
\usepackage{mathtools,leftidx}
\usepackage{enumitem}
\usepackage{amsmath}
\usepackage{mathrsfs}
\usepackage{accents}
\usepackage{scalerel}
\usepackage{aligned-overset}
\usepackage{tcolorbox}
\setlist[enumerate,2]{label={(\theenumi.\theenumii)},ref={(\theenumi.\theenumii)}}
\setlist[enumerate,1]{label={(\theenumi)},ref={(\theenumi)}}

\usetikzlibrary{svg.path,arrows,calc}
\tikzset{>=latex'}

\usepackage[toc]{appendix}
\usepackage{etoolbox}
\cslet{blx@noerroretextools}\empty
\usepackage[maxbibnames=9,giveninits=true,style=alphabetic,sorting=nyt,backend=biber]{biblatex}
\DeclareFieldFormat{titlecase:title}{\MakeSentenceCase*{#1}}
\renewbibmacro*{url+urldate}{%
\iffieldundef{urlyear}
  {}
  {}
\usebibmacro{url}}
\def\restrict#1{\raise-.5ex\hbox{\ensuremath|}_{#1}}

\usepackage[colorlinks, citecolor=citegreen, linkcolor=refred,urlcolor=blue, unicode,psdextra,hypertexnames=false]{hyperref}
\usepackage{cleveref}
\crefname{enumi}{}{}
\crefname{enumii}{}{}
\expandafter\def\csname ver@etex.sty\endcsname{3000/12/31}

\makeatletter
 \def\author@andify{
 \nxandlist {\unskip{\kern.3cm} \penalty-2}
 {\unskip {\kern.3cm} \penalty-2}
 {\unskip {\kern.3cm} \penalty-2}}
\makeatother

\newcommand{\orcid}[1]{\unskip {} \raisebox{.4ex}{\href{https://orcid.org/#1}{\resizebox{.25cm}{.25cm}{\includegraphics{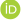}}}}}

\usepackage{accents}
\usepackage{autonum}
\usepackage{gasymbols}
\usepackage{nicefrac}

\definecolor{citegreen}{rgb}{0,0.3,0}
\definecolor{refred}{rgb}{0.5,0,0}

\def\mathring#1{\accentset{\circ}{#1}}
\usepackage[equation=section]{theorems}

\let\oldemail\email
\let\email\relax
\def\email#1{\oldemail{\href{mailto:#1}{\textcolor{black}{#1}}}}

\title[Nonlinear potential theory and Ricci-pinched $3$-manifolds]{Nonlinear potential theory and Ricci-pinched $3$-manifolds}

\author[L.~Benatti]{Luca Benatti\orcid{0000-0002-4685-7443}}
\address{L.~Benatti, Universit\`a di Pisa,
Largo Bruno Pontecorvo 5, 56127 Pisa, Italy}
\email{luca.benatti@dm.unipi.it}

\author[A.~León Quirós]{Ariadna León Quirós}
\address{A.~León Quirós, Eberhard Karls Universit\"{a}t Tübingen, Fachbereich Mathematik, Auf der Morgenstelle 10, 72076 T\"{u}bingen, Germany}
\email{quiros@math.uni-tuebingen.de}

\author[F.~Oronzio]{Francesca Oronzio\orcid{0009-0006-6597-0077}}
\address{F.~Oronzio, Scuola Superiore Meridionale, Via Mezzocannone 4, 80134 Napoli, Italy}
\email{f.oronzio@ssmeridionale.it}

\author[A.~Pluda]{Alessandra Pluda\orcid{0000-0003-4714-4119}}
\address{A.~Pluda, Universit\`a di Pisa,
Largo Bruno Pontecorvo 5, 56127 Pisa, Italy}
\email{alessandra.pluda@unipi.it}

\renewcommand{\ncapa}{\mathfrak{c}}


\newcommand{\MF}{\mathscr{F}}
\newcommand{\IF}{\mathscr{G}}

\newcommand{\sml}[1]{\scaleobj{.9}{(#1)}}

\begin{document}

 \begin{abstract}
In this paper, we focus on {\em Hamilton's pinching conjecture} formulated in \cite{Hamilton1982ThreemanifoldsWP}. Let $(M, g)$ be a complete, connected, noncompact Riemannian $3$-manifold satisfying the {\em Ricci-pinching condition}. Then, it is flat. Here, we give an alternative proof, based on nonlinear potential theory, under the extra hypothesis of superquadratic volume growth.
\end{abstract}  
\maketitle


\smallskip

\noindent \underline{\smash{Keywords}}: Ricci-pinched Riemannian $3$-manifold, monotonicity formulas, nonlinear potential theory.

\section{Introduction}
A Riemannian manifold $(M,g)$ is said to be {\em Ricci-pinched} if $\Ric\geq 0$ and there exists a constant $\varepsilon>0$ such that $\Ric\geq\varepsilon\mathrm{R}g$, where $\Ric$ and $\mathrm{R}$ are the Ricci and scalar curvature of $(M,g)$, respectively. Ricci-pinched manifolds are the focus of {\em Hamilton's pinching conjecture}, which was stated by Richard Hamilton in \cite{Hamilton1982ThreemanifoldsWP} and is now established as a theorem:
\begin{theorem}\label{thmpinchingtheorem}
Let $(M, g)$ be a complete, connected Riemannian $3$-manifold. Suppose that $M$ is Ricci-pinched. Then, $M$ is flat or compact.
\end{theorem}

The theorem can be interpreted as an extension of Bonnet--Myers' theorem in three dimensions: if $(M,g)$ is a complete and connected $n$-dimensional Riemannian manifold such that $\Ric\geq (n-1)k^2g$, for some constant $k>0$, then $M$ is compact and $\mathrm{diam}(M,g)\leq\pi/k$. 

There is also an extrinsic counterpart of the theorem for hypersurfaces of the Euclidean space, that was proved by Hamilton himself in~\cite{hamilton_1994}.

\begin{theorem}
Let $M^n$ be a smooth, strictly convex, and complete hypersurface bounding a region in $\R^{n+1}$. Suppose that its second fundamental form is $\varepsilon$-pinched, in the sense that
\begin{equation}
\h_{ij}\geq\varepsilon\H g_{ij},
\end{equation}
where $ g_{ij}$ is the induced Riemannian metric, $\h$ the second fundamental form and its trace $\H$ is the mean curvature, for some $\varepsilon>0$. Then $M^n$ is compact. 
\end{theorem}

Regarding the intrinsic version of the Ricci-pinching theorem, the first step towards the proof was taken by Chen and Zhu \cite{chenzhu_completeriemannianmanifolds_2000}. They showed that a complete noncompact Riemannian $3$-manifold, which has bounded and nonnegative sectional curvature and is Ricci-pinched, is flat. Then, Lott \cite{lott_3manifolds_2023} enhanced their result, requiring less stringent conditions on the sectional curvature. Deruelle--Schulze--Simon \cite{deruelle_initialstabilityRicciflow_2022} proved that the conjecture holds if the curvature is merely bounded. Finally, Lee and Topping \cite{leetopping_3manifoldnonnegpinchedRic_2022} removed the assumption on the curvature. 

\medskip

All the above mentioned results rely on the Ricci flow. A proof based on different geometric flows is conceivable; however, current attempts in the literature require an extra assumption on the volume growth at infinity \cite{huisken_inversemeancurvatureflow_2023, benatti_notericcipinchedthreemanifolds_2024}. This is also the case in the present paper. We say that a complete, connected, noncompact Riemannian $3$-manifold $(M, g)$  with nonnegative Ricci curvature has \emph{superlinear} volume growth if there exist $q\in M$, $\kst_{\mathrm{vol}}> 0$, $\tilde{r}>0$ and $\alpha>0$ such that for every radius $r>\tilde{r}$ it holds
\begin{equation}\label{superquadratic}
\kst_{\mathrm{vol}}^{-1}\, r^{1+\alpha}\leq \abs{B_r(q)}\leq \kst_{\mathrm{vol}}\, r^{1+\alpha}.
\end{equation}
By Bishop--Gromov theorem the parameter $\alpha$ is in $(0, 2]$, and it is independent of the point $q\in M$, while the constant $\kst_{\mathrm{vol}}$ could depend on it. We say that $M$ has {\em superquadratic} volume growth if $\alpha\in (1, 2]$ and {\em Euclidean volume growth} if $\alpha=2$. In the latter case, $(M,g)$ has strictly positive {\em asymptotic volume ratio} 
\begin{equation}
\AVR(g)\coloneqq\frac{3}{4\pi}\lim_{r\to +\infty}\frac{\abs{B_r(p)}}{r^3},\qquad\text{with}\;p\in M.
\end{equation} 

The aim of this short note is to provide a new proof of the following result.
\begin{theorem}\label{main}
Let $(M,g)$ be a complete, connected, noncompact, Ricci–pinched Riemannian $3$–manifold. Suppose that $(M,g)$ has superquadratic volume growth. Then, $(M,g)$ is  isometric to $\R^3$.
\end{theorem} 

\cref{main} is contained in Deruelle--Schulze--Simon~\cite[Theorem 1.3]{deruelle_initialstabilityRicciflow_2022} and proved by Huisken--Körber~\cite[Theorem 12]{huisken_inversemeancurvatureflow_2023} employing the inverse mean curvature flow.  The experts will quickly realize that our proof resembles the one by Huisken--Körber, that is intimately based on the monotonicity of the Willmore functional along the inverse mean curvature flow in Ricci-pinched manifolds. Here, we replace the inverse mean curvature flow with $p$-harmonic potentials with $p\in (1,2)$ and the Willmore functional with a suitable proxy defined in \cref{monotone_F}. Crucial in \cite{huisken_inversemeancurvatureflow_2023} is the use of both the Gauss equation and Gauss--Bonnet theorem. The main technical difficulty of our note is the lack of the regularity required to define the Euler characteristic and of the topological properties of the level sets and thus of Gauss--Bonnet theorem. We overcome this issue thanks to a weak version of the theorem recently shown in \cite{benatti_finepropertiesnonlinearpotentials_2024}.

It should be noted that \cref{main} is a generalization of \cite[Theorem 1.5]{benatti_notericcipinchedthreemanifolds_2024}, which instead requires $\alpha>4/3$ in \cref{superquadratic}. The extra assumption on $\alpha$ depends on the technique employed. The proof in \cite{benatti_notericcipinchedthreemanifolds_2024} uses the harmonic potential, which corresponds to the case $p=2$. Here we have the freedom to vary the values of $p$. For each $p$ we obtain a different lower bound on $\alpha$ which approaches $1$ in the limit as $p\to 1^+$. On the other hand, \cite{benatti_notericcipinchedthreemanifolds_2024} has the merit of exploiting the higher regularity of harmonic functions allowing for less technical arguments. This simplification somehow balances the stronger assumption.

We conclude our note by considering the case for manifolds with boundary in \cref{thm:boundary}.

\subsection{Acknowledgments}  
L.B. and A.P. are partially supported by the BIHO Project ``NEWS - NEtWorks and Surfaces evolving by curvature'' and by the MUR Excellence Department Project awarded to the Department of Mathematics of University of Pisa. L.B. is partially supported by the PRIN Project 2022PJ9EFL ``Geometric Measure Theory: Structure of Singular Measures, Regularity Theory and Applications in the Calculus of Variations''.  A.P. is partially supported by the PRIN Project 2022R537CS ``$\rm{NO}^3$ - Nodal Optimization, NOnlinear elliptic equations, NOnlocal geometric problems, with a focus on regularity'' and PRA2022 ``GEODOM''. L.B., F.O and A.P. are members of INdAM - GNAMPA.

A.L.Q. is thankful to her supervisor Gerhard Huisken for his guidance and to Albachiara Cogo for the fruitful conversations on the topic of this paper.

The authors are grateful to the anonymous referees for their insightful comments and constructive feedback, which significantly improved the paper.

\section{Preliminaries}
\begin{center}
\begin{tcolorbox}[frame empty, colback=gray!20, halign=center, width=.8\textwidth, sharp corners,boxsep=1mm, before skip=.2cm,after skip=.1cm]
    \textit{We assume throughout the paper that $(M,g)$ is a smooth, complete, connected,}\\ \textit{orientable, and noncompact Riemannian $3$-manifold.}
\end{tcolorbox}
\end{center}

\subsection{Basic results} Let $(M, g)$ be 
with nonnegative Ricci and superquadratic volume growth
and take $p\in (1,2)$. For a given closed and bounded set $\Omega\subset M$ with smooth boundary, we define the function $w_p$ as the solution to
\begin{equation}\label{eq:moser_p_potential}
\begin{cases}
\Delta_p w_p& =&\abs{\nabla w_p }^p &\text{on $ M\smallsetminus\Omega $}\\
w_p&=&0&\text{on $\partial\Omega$}\\
w_p&\to&+\infty &\text{as $ d(x,o)\to+\infty$}
\end{cases}
\end{equation}
where $\Delta_p f=\div\left( \abs{\nabla f}^{p-2} \nabla f \right)$ denotes the $p$-Laplacian operator of $(M,g)$. 

Existence, uniqueness and regularity of such a solution come from the classical theory for the slightly different problem
\begin{equation}\label{eq:p_harmonic_potential}
\begin{cases}
\Delta_p u_p& =& 0 &\text{on $ M\smallsetminus\Omega $}\\
u_p&=&1&\text{on $\partial\Omega$}\\
u_p&\to&0 &\text{as $ d(x,o)\to+\infty$}.
\end{cases}
\end{equation}
Specifically, the function $w_p=-(p-1)\log u_p$ is a weak solution of the problem \cref{eq:moser_p_potential} provided $u_p$ solves \cref{eq:p_harmonic_potential}. Hence, the study of the properties of $w_p$ reduces to that of $u_p$.

 Problem \cref{eq:p_harmonic_potential} has a unique weak solution, which takes values in $(0,1]$. The function $u_p\in W^{1,p}_{\mathrm{loc}}(M\smallsetminus \Omega)$ is of class $\mathscr{C}^{1,\beta}$ in any precompact set $K$, for some $\beta>0$ depending on $K$. It is smooth on the open set $\left\lbrace\abs{\nabla u_p} \neq 0\right\rbrace$ and smoothly attains the boundary value on $\partial \Omega$ (see~\cite{holopainen_volumegrowthgreenfunctions_1999} and~\cite[Proposition 2.3.5 and Proposition 2.3.6]{benatti_monotonicityformulasnonlinearpotential_2022}). Moreover,  $\abs{\nabla u_p}^{p-1}\in W^{1,2}_{\loc}$ (see \cite{lou_singularsetslocalsolutions_2008}). 

On manifolds with nonnegative Ricci curvature, Laplacian comparison theorem yields a \textit{lower bound} for $u_p$, namely
\begin{equation}\label{lower_bound}
    u_p(x) \geq \kst\, \dist(x,o)^{-\frac{3-p}{p-1}}.
\end{equation}
for some positive constant $\kst= \kst(M, \Omega, p)>0$. Under the additional assumption of superquadratic growth \cref{superquadratic}, we also have an \textit{upper bound} for $u_p$ that reads as
\begin{equation}\label{upper_bound}
u_p(x) \leq \kst\,\dist(x,o)^{-\frac{\alpha+1-p}{p-1}},
\end{equation}
for some constant $\kst= \kst(M, \Omega, p)>0$. For a detailed proof of these two results, we refer the reader to \cite[Theorem 2.11]{benatti_minkowskiinequalitycompleteriemannian_2024} and \cite[Proposition 5.10]{holopainen_volumegrowthgreenfunctions_1999} (see also \cite[Proposition 2.3.2]{benatti_monotonicityformulasnonlinearpotential_2022}).

Let $\Omega^{\sml{p}}_t=\left\lbrace w_p\leq t\right\rbrace\cup \Omega$, then for almost every $t\in [0,+\infty)$, the set $\partial\Omega_t^{\sml{p}}\cap \left\lbrace\abs{\nabla w_p}=0\right\rbrace$ is $\Hff^{2}$-negligible. Almost every $\partial \Omega^{\sml{p}}_t$ admits a suitable geometric notion of (weak) mean curvature $\H$ and second fundamental form $\h$ (see \cite{benatti_finepropertiesnonlinearpotentials_2024}). 
\medskip

We recall the definition of the normalized $p$-capacity of $D$  closed bounded subset of $M$
\begin{equation}
\ncapa_p(\partial D) =\inf\left\lbrace\frac{1}{4\pi }\left(\frac{p-1}{3-p}\right)^{p-1}\int_{M\smallsetminus D}\abs{\nabla\psi}^p\dif\Hff^{2}\,\,\bigg\vert\,\,\psi\in\CS^\infty_c(M),\psi\geq\chi_D\right\rbrace.
\end{equation}
Note that $\ncapa_p(\S^2) = 1$.

\begin{lemma}
For almost every $t\in [0,+\infty)$ we have
\begin{equation}\label{eq:ncapa_def}
\ncapa_p(\partial\Omega^{\sml{p}}_t)=\frac{1}{4\pi}\int_{\partial\Omega_t^{\sml{p}}}\left(\frac{\abs{\nabla w_p}}{3-p}\right)^{p-1}\dif\Hff^{2},
\end{equation}
and
\begin{equation}\label{eq:capa}
\ncapa_p(\partial\Omega^{\sml{p}}_t) =\ee^t\ncapa_p(\partial\Omega).\end{equation} 
\end{lemma}
We refer to \cite[Proposition 3.2.1]{tolksdorf_dirichletproblemquasilinearequations_1983} and \cite[Propositions 2.8 and 2.9]{benatti_minkowskiinequalitycompleteriemannian_2024} for the proof of these facts.

\subsection{Monotone quantities}
Let $w_p$ be the solution of the problem \cref{eq:moser_p_potential}, let $\Omega^{\sml{p}}_t=\left\lbrace w_p\leq t\right\rbrace\cup \Omega$ and denote by $\H$ the (weak) mean curvature with respect to the outward-pointing unit normal $\nu=\nabla w_{p}/\abs{\nabla w_p}$. For almost every $t\in [0,+\infty)$ we set
\begin{align}
    \MF_p(t)&= \int_{\partial\Omega^{\sml{p}}_t}\frac{\H^2}{4}-\left(\frac{\H}{2}-\frac{\abs{\nabla w_p}}{(3-p)}\right)^2\dif\Hff^{2},\label{monotone_F}\\
    \IF_p(t)&=\int_{\partial\Omega^{\sml{p}}_t}\frac{\abs{\nabla w_p}^{2}}{(3-p)^2}\dif\Hff^{2}.
\end{align}

Both $\MF_p(t)$ and $\IF_p$ are well-defined $L^1$-functions (for a detailed explanation see \cite[Remark 3.2 and Section 3.3]{benatti_finepropertiesnonlinearpotentials_2024}).
Observe that $\IF_p(t)\geq0$ and it is bounded on a manifold with nonnegative Ricci curvature. Indeed, by \cite[Theorem A]{wang_local_2010} we have a Cheng--Yau gradient bound for $w_p$: there exists a positive constant $\kst$ such that $\abs{\nabla w_p}(x) \leq \kst \, \dist(x,o)^{-1}$. Hence, by \cref{lower_bound} and \cref{eq:capa} we have
\begin{equation}
    \IF_p(t) \leq  \frac{4\pi\ncapa_p(\partial \Omega_t^{\sml{p}})}{(3-p)^{3-p}} \,\sup_{\partial \Omega_t^{\sml{p}}}\abs{\nabla w_p}^{3-p} \leq \kst\ee^{-t}\, \ncapa_p(\partial \Omega_t^{\sml{p}})  = \kst \,\ncapa_p(\partial \Omega).
\end{equation}

\begin{lemma}\label{lem:monotonicity}
Let $(M, g)$ be a Riemannian $3$-manifold and let $w_p$ be a solution to problem \eqref{eq:moser_p_potential}. Then, $\MF_p\in W^{1,1}_{\loc}(0,+\infty)$ and $\IF_p\in W^{2,1}_{\loc}(0,+\infty)$. Moreover, for almost every $t\in [0,+\infty)$, the following holds
    \begin{align}
        \MF^{'}_p(t)&=-\frac{1}{3-p}\int_{\partial\Omega^{\sml{p}}_t}\Ric(\nu,\nu)+\abs{\mathring{\h}}^{2}+\frac{\abs{\nabla^{\top}\abs{\nabla w_p}}^{2}}{\abs{\nabla w_p}^{2}}+\frac{3-p}{2(p-1)}\left(\H-\frac{2}{3-p}\abs{\nabla w_p}\right)^2d\Hff^{2}\label{eq:deriv_F}\\
        \IF_{p}^{'}(t)&=\frac{1}{p-1}\int_{\partial\Omega^{\sml{p}}_t}\frac{2\abs{\nabla w_p}^{2}}{(3-p)^2}-\frac{\H\abs{\nabla w_p}}{3-p} \dif\Hff^{2},\label{eq:deriv_G}
    \end{align}
    where  $\h$ is the second fundamental form of $\partial \Omega_{t}^{(p)}$, $\mathring{\h}$ is the traceless part of $\h$ and $\nabla^{\top}$ denotes the tangential part of the gradient with respect to $\partial \Omega^{\sml{p}}_t$. 
\end{lemma}
\begin{proof}
[Idea of the proof]
For the well-posedness and regularity of $\MF_p$ and $\IF_p$ we refer, for instance, to \cite[Section 3.3]{benatti_finepropertiesnonlinearpotentials_2024} and to \cite[Theorem 3.1 and Proposition 3.5]{benatti_minkowskiinequalitycompleteriemannian_2024}.  We sketch the (formal) computations that lead to \cref{eq:deriv_F} and \cref{eq:deriv_G}. Call 
\begin{align}
X&=\frac{\abs{\nabla w_p}\nabla w_p}{(3-p)^2},\\
Y&=\frac{1}{3-p}\left(\frac{\Delta w_p\nabla w_{p}}{\abs{\nabla w_p}}-\frac{\nabla\abs{\nabla w_p}^{2}}{2\abs{\nabla w_p}}\right)-X.
\end{align}
Then, by direct computation,
\begin{align}
        \div(X)&=\frac{\abs{\nabla w_p}}{p-1}\left(\frac{2\abs{\nabla w_p}^{2}}{(3-p)^2}-\frac{\H }{3-p}\right),\\
        \div(Y)&=
        -\frac{\Ric(\nabla w_p,\nabla w_p)}{\abs{\nabla w_p}}-\abs{\mathring{\h}}^{2}\abs{\nabla w_p}-\frac{\abs{\nabla^{\top}\abs{\nabla w_p}}^{2}}{\abs{\nabla w_p}}-\frac{3-p}{2(p-1)}\abs{\nabla w_p}\left(\H-\frac{2}{3-p}\abs{\nabla w_p}\right)^{2}.
\end{align}
    Then, we can write 
\begin{align}
   \IF_{p}(t)=\int_{\partial\Omega^{\sml{p}}_t} \ip{X|\frac{\nabla w_p}{\abs{\nabla w_p}}} \dif\Hff^{2},\qquad \MF_{p}(t)=\int_{\partial \Omega^{\sml{p}}_t}\ip{Y|\frac{\nabla w_p}{\abs{\nabla w_p}}}\dif\Hff^{2}.
\end{align}    
Hence for $s<t$ in $[0,+\infty)$ by the divergence theorem $\MF_{p}(t)-\MF_{p}(s)=\int_{\set{ s<w<t}}\div(Y)\dif\Hff^{2}$ and $\IF_{p}(t)-\IF_{p}(s)=\int_{\set{ s<w<t}}\div(X)\dif\Hff^{2}$. Thus, at least formally, we obtain \cref{eq:deriv_F} and \cref{eq:deriv_G}.
To give a precise meaning to this formal computation we need quite a bit of extra work. To this aim, we refer to \cite[Theorem 3.1]{benatti_finepropertiesnonlinearpotentials_2024}.  
\end{proof}

\begin{lemma}\label{lem:monotonicity2}
    Let $(M, g)$ be a Riemannian $3$-manifold with nonnegative Ricci curvature and let $w_p$ be a solution to problem \eqref{eq:moser_p_potential}. Both $\MF_p (t)$ and $\IF_p(t)$ are essentially monotone nonincreasing. Moreover, for almost every $t \in [0+\infty)$ we have  $0\leq \IF_p(t) \leq \MF_p(t)$  and 
    \begin{equation}\label{eq:derGin0}
        \IF_p'(0)= \frac{1}{p-1}\int_0^{+\infty} e^{-\frac{t}{p-1}}\MF_p'(t) \dif t.
    \end{equation}
\end{lemma}

\begin{proof}
If $\Ric \geq 0$, $\MF_p$ is manifestly monotone nonincreasing by \cref{lem:monotonicity}. To justify the sign of $\IF'_p$, we argue as in \cite{benatti_minkowskiinequalitycompleteriemannian_2024}. Observe that
\begin{equation}\label{eq:zzzderivative}
    (\ee^{-\frac{t}{p-1}}\IF_p'(t))' = - \frac{1}{p-1} \ee^{-\frac{t}{p-1}} \MF_p'(t)\geq 0.
\end{equation}
Hence, for every $s\leq t$ one has $\IF_p'(t)\geq \ee^{-\frac{s-t}{p-1}} \IF_p'(s)$, that integrated in $t$ over $[s,r]$ gives
\begin{equation}\label{eq:zzzintegratedGder}
    \IF_p(r)-\IF_p(s) \geq (p-1)\left( \ee^{-\frac{s-r}{p-1}} -1\right) \IF_p'(s).
\end{equation}
The existence of a $s\geq 0$ for which $\IF_p'(s)>0$ would be in contradiction to the boundedness of $\IF_p(t)$. Therefore, 
    \begin{equation}
         0\geq \IF_{p}^{'}(t)\overset{\cref{eq:deriv_G}}{=}\frac{1}{p-1}\left(\IF_{p}(t)-\MF_{p}(t)\right)
    \end{equation}
and in particular $0\leq \IF_p(t) \leq \MF_p(t)$ for almost every $t\in [0,+\infty)$.

Observe that $\ee^{-\frac{t}{-p-1}}\IF_p'(t)$ is nonpositive and monotone nondecreasing. If it does not vanish at infinity, we would have $\IF_p'(t) \leq -\kst \ee^{\frac{t}{p-1}}$ for some constant $\kst>0$. Integrating it would contradict that $\IF_p(t)$ is bounded from below. In conclusion, \cref{eq:derGin0} follows integrating \cref{eq:zzzderivative} on $[0,+\infty)$.
\end{proof}

\section{Proof of the main theorem}

\subsection{Weak Gauss-Bonnet theorem}

If $\Sigma$ is a smooth closed surface in $M$ 
with normal $\nu$, mean curvature $\H$ and traceless second fundamental form $\mathring{\h}$,  it holds
\begin{align} 
2\int_{\Sigma}\Ric(\nu,\nu)\dif\Hff^2
 &\geq\varepsilon\left(16\pi-\int_{\Sigma}\H^2\dif\Hff^2\right) &\quad\text{if}\;\mathrm{genus}(\Sigma)=0,\\
 2\int_{\Sigma}\Ric(\nu,\nu)+\abs{\mathring{\h}}^2\dif\Hff^2&\geq\int_{\Sigma}\H^2\dif\Hff^2&\quad\text{if}\;\mathrm{genus}(\Sigma)\geq 1.
\end{align}

These two inequalities follow combining the Gauss equation 
\begin{equation}\label{eq:Gauss}
      \sca^\top = \sca - 2 \Ric(\nu,\nu) + \H^2 - \abs{\h}^2
\end{equation}
with the Gauss--Bonnet Theorem, and the pinching condition \cite[Lemma 8]{huisken_inversemeancurvatureflow_2023}. These inequalities are crucial both in \cite{huisken_inversemeancurvatureflow_2023} and in \cite{benatti_notericcipinchedthreemanifolds_2024}.  Note that in both papers the considered flow is regular enough to perform the desired estimates. Indeed, in the former case, the level sets of the weak inverse mean curvature flow are $C^{1,\beta}$, and they can be approximated in $W^{2,2}$ by smooth surfaces. In the latter case, harmonic functions are smooth and Sard's theorem applies. Unfortunately, in the current situation we cannot infer the regularity of the level sets required to apply the classical Gauss--Bonnet theorem. 

We recall that the set $\partial\Omega^{\sml{p}}_t\cap \left\lbrace\abs{\nabla w_p}=0\right\rbrace$ is $\Hff^{2}$-negligible. Outside the critical set, $w_p$ is smooth and so the induced scalar curvature is well-defined. On the other hand, all terms appearing in the right-hand side of \cref{eq:Gauss} are well-defined integrable functions on the level sets of $w_p$ thanks to \cite[Theorem 3.1]{benatti_finepropertiesnonlinearpotentials_2024}.  Hence, \cref{eq:Gauss} can be adopted as a definition of the induced scalar curvature $\sca^\top$ on the whole level set. The following statement serves as a weak replacement for the classic Gauss-Bonnet theorem. It states that the integral of $\sca^\top$ is essentially discrete with values in the same set that the classical Gauss--Bonnet theorem would give. For a proof of the theorem, we refer the reader to \cite[Theorem 1.3]{benatti_finepropertiesnonlinearpotentials_2024}.

\begin{theorem}
     Let $(M,g)$ be a Ricci-pinched $3$-dimensional Riemannian manifold with superquadratic volume growth. Let $p \in (1,2)$, $w_p$ be the proper solution to \cref{eq:moser_p_potential}. Then, for almost every $t \in [0,+\infty)$  it holds
    \begin{equation}
        \int_{\partial \Omega^{\sml{p}}_t} \sca^{\top} \dif \Hff^2 \in 8 \pi \Z.
    \end{equation}
      \end{theorem}

\begin{lemma}
    Let $(M,g)$ be a Ricci-pinched $3$-dimensional Riemannian manifold with superquadratic volume growth. Let $p \in (1,2)$, $w_p$ be the proper solution to \cref{eq:moser_p_potential}. Then, for almost every $t \in [0,+\infty)$  it holds
\begin{align} 
2\int_{\partial \Omega^{\sml{p}}_t}\Ric(\nu,\nu)\dif\Hff^2
 &\geq\varepsilon\left(16\pi-\int_{\partial \Omega^{\sml{p}}_t}\H^2\dif\Hff^2\right) &&\text{if }\int_{\partial \Omega^{\sml{p}}_t} \sca^\top\dif\Hff^{2}\geq 8\pi,\label{eq:genere_zero}\\
 2\int_{\partial \Omega^{\sml{p}}_t}\Ric(\nu,\nu)+\abs{\mathring{\h}}^2\dif\Hff^2&\geq\int_{\partial \Omega^{\sml{p}}_t}\H^2\dif\Hff^2&&\text{if }\int_{\partial \Omega^{\sml{p}}_t} \sca^\top\dif\Hff^{2}\leq 0.\label{eq:genere_maggiore_uno}
\end{align}
\end{lemma}
\begin{proof}
We can write the Gauss equation \cref{eq:Gauss} in integral from, that is
    \begin{equation}
        \int_{\partial \Omega^{\sml{p}}_t} \H^2\dif\Hff^2= \int_{\partial \Omega^{\sml{p}}_t} 2\sca^\top-2\sca+4\Ric(\nu,\nu)+2\abs{\mathring{\h}}^2\dif\Hff^2.
    \end{equation}
Since $\Ric\geq 0$, it follows that $\Ric(\nu,\nu)-\sca\leq 0$. Then, if $\int_{\partial \Omega^{\sml{p}}_t} \sca^\top\leq 0$ we trivially get
\begin{equation}
  \int_{\partial \Omega^{\sml{p}}_t} \H^2\dif\Hff^2=  \int_{\partial \Omega^{\sml{p}}_t} 2\sca^\top-2\sca+2\Ric(\nu,\nu)+2\Ric(\nu,\nu)+2\abs{\mathring{\h}}^2\dif\Hff^2\leq   \int_{\partial \Omega^{\sml{p}}_t} 2\Ric(\nu,\nu)+2\abs{\mathring{\h}}^2\dif\Hff^2.
\end{equation}
Conversely, if $\int_{\partial \Omega^{\sml{p}}_t} \sca^\top\geq 8\pi$, using the pinching condition, we obtain
    \begin{equation}
      \int_{\partial \Omega^{\sml{p}}_t} \H^2\dif\Hff^2=   \int_{\partial \Omega^{\sml{p}}_t} 2\sca^\top-2\sca+4\Ric(\nu,\nu)+2\abs{\mathring{\h}}^2\dif\Hff^2\geq 16\pi -\frac{1}{\varepsilon} \int_{\partial \Omega^{\sml{p}}_t} 2\Ric(\nu,\nu)\dif\Hff^2,
    \end{equation}
 as desired.   
\end{proof}

\subsection{Asymptotic behavior of $\MF_p$} 
Let $w_p$ be a solution to \eqref{eq:moser_p_potential} for a set $\Omega$ with smooth boundary. In the following lemma, we exhibit the asymptotic behavior of $\MF_p$ under the assumption $\MF_p(0)<4\pi$. 

\begin{lemma}\label{lemma 2}
   Let $(M,g)$ be a Ricci-pinched $3$-dimensional Riemannian manifold with superquadratic volume growth. Assume $\MF_p(0)<4\pi$. There exist $T_0\in[0,+\infty)$ and a positive constant $\kst=\kst(T_0)$ such that for almost every $t\geq T_0$, there holds $\MF_{p}(t)\leq \kst\ee^{-\frac{2}{3-p}t}$.
\end{lemma}
\begin{proof}
    For almost every $t\in [0,+\infty)$, we know that $\int_{\partial \Omega^{\sml{p}}_t} \sca^{\top} \dif \Hff^{2} \in 8 \pi \Z$. If $\int_{\partial\Omega^{\sml{p}}_t} \sca^{\top}\dif\Hff^{2} \leq 0$, combining the expression \cref{eq:deriv_F} with the estimate \cref{eq:genere_maggiore_uno}, we obtain
    \begin{align}
        -2(3-p)\MF_p^{'}(t)\geq 2\int_{\partial\Omega^{\sml{p}}_t}\Ric(\nu,\nu)+\abs{\mathring{\h}}^{2}\dif\Hff^{2}\geq \int_{\partial\Omega^{\sml{p}}_t}\H^{2}\dif\Hff^{2}\geq 4\MF_{p}(t).
    \end{align}
    If $\int_{\partial\Omega^{\sml{p}}_t} \sca^{\top}\dif\Hff^{2} \geq 8\pi$, by \cref{eq:deriv_F} and the estimate \cref{eq:genere_zero}, we have 
    \begin{align}
        -2(3-p)\MF_{p}^{'}(t)&\geq \int_{\partial \Omega^{\sml{p}}_t}2 \Ric(\nu,\nu)+\frac{3-p}{p-1}\left(\H-\frac{2}{3-p}\abs{\nabla w_p}\right)^2\dif\Hff^{2}\\
        &\geq \int_{\partial\Omega^{\sml{p}}_t}2\Ric(\nu,\nu)+\varepsilon \left(\H-\frac{2}{3-p}\abs{\nabla w_p}\right)^{2}\dif\Hff^{2}\\
        &\geq \varepsilon\left(16\pi-\frac{4}{3-p}\int_{\Omega^{\sml{p}}_t}\H\abs{\nabla w_p}-\frac{\abs{\nabla w_p}^{2}}{3-p}\dif\Hff^{2}\right)\\
        &\geq \varepsilon\left(16\pi-4\MF_{p}(t)\right)
    \end{align}
    where, up to decreasing $\varepsilon$, we can suppose $\varepsilon\leq \frac{1}{3}<\frac{3-p}{p-1}$.
     Then, we conclude that
    \begin{align}
        \MF_p^{'}(t)\leq \max\left\lbrace-\frac{2}{3-p}\MF_p(t),-\frac{2\varepsilon}{3-p}\left(4\pi-\MF_p(t)\right)\right\rbrace
    \end{align}
    holds for almost every $t\in[0,+\infty)$. 

    If $\MF_p(t) \geq \varepsilon (4 \pi - \MF_p(t))$ for all $t\geq 0$, we would have $\MF_p(t) \geq 4\pi\varepsilon/(1+\varepsilon)$ and
        \begin{equation}
            \MF_{p}^{'}(t)\leq -\frac{2\varepsilon }{3-p}\left(4\pi- \MF_{p}(t)\right)\leq -\frac{2\varepsilon }{3-p}\left(4 \pi -\MF_{p}(0)\right) <0,
        \end{equation}
        which is not possible. Therefore, there must be a $T_0\geq 0 $ such that $\MF_p(T_0) \leq \varepsilon (4 \pi - \MF_p(T_0))$. Since $\MF_p$ is a $W^{1,1}_{\loc}$ function and it is monotone nonincreasing $\MF_p(t) \leq \varepsilon (4 \pi - \MF_p(t))$ must hold for all $t \geq T_0$, implying $\MF'_p(t) \leq -\frac{2}{3-p}\MF_p(t)$ and consequently $\MF_{p}(t)\leq \kst(T_0) \ee^{-\frac{2}{3-p}t}$ for all $t \geq T_0$.
\end{proof}

\subsection{Proof of \cref{main}}
Assume $(M,g)$ is orientable and there exists a set $\Omega$ and a solution to \cref{eq:moser_p_potential} such that $\MF_p(0)<4\pi$.
By equations \cref{eq:capa} and \cref{eq:ncapa_def} and H\"older inequality, for almost every $t\in [0,+\infty)$, we obtain 
\begin{align}
 \ee^{t}\ncapa_{p}(\partial\Omega^{\sml{p}})\overset{\cref{eq:capa}}&{=}\ncapa_{p}(\partial\Omega_{t}^{(p)})\overset{\cref{eq:ncapa_def}}{=}\frac{1}{4\pi}\int_{\partial\Omega^{\sml{p}}_t}\left(\frac{\abs{\nabla w_p}}{3-p}\right)^{p-1}\dif\Hff^2 \\
 &\leq  \frac{1}{4\pi}\frac{1}{(3-p)^{p-1}}\left(\int_{\partial\Omega^{\sml{p}}_t}\abs{\nabla w_p}^{2}\dif\Hff^2\right)^{\frac{p}{3}}\left(\int_{\partial\Omega^{\sml{p}}_t }\abs{\nabla w_{p}}^{-1}\dif\Hff^2\right)^{\frac{3-p}{3}}.
\end{align}
From \cref{lem:monotonicity2,lemma 2}, we know that there exists a $T_{0}\in [0,+\infty)$ and a positive constant $\kst$ such that for almost all $t\in[T_{0},+\infty)$, the following holds
\begin{equation}
    \int_{\partial\Omega^{\sml{p}}_t}\frac{\abs{\nabla w_p}^{2}}{(3-p)^2}\dif\Hff^{2}=\IF_{p}(t)\leq \kst\ee^{-\frac{2}{3-p}t}.
\end{equation}
Thus, by the coarea formula, we obtain 
\begin{equation}
    \frac{\dif}{\dif t}\Vol\left(\left\lbrace w_p\leq t\right\rbrace\smallsetminus \left\lbrace\abs{\nabla w_p}=0\right\rbrace\right)=\int_{\partial\Omega^{\sml{p}}_t}\abs{\nabla w_p}^{-1}d\Hff^{2}\geq 
    \kst\ee^{\frac{9-p}{(3-p)^{2}}t}
\end{equation}
for almost every $t\in[0,+\infty)$. 

Consider $R_{t}=\sup\left\lbrace\dist (x,o):w_p(x) \leq t \right\rbrace$ for any $t\in [0,+\infty)$. Let $T_{1}\in (T_{0},+\infty)$. By integrating the above inequality over $[T_{0},T_{1}]$ and using the superquadratic volume growth condition, we get
\begin{equation} \label{eq:inequality}
    \kst\left(\ee^{\frac{9-p}{(3-p)^{2}}T_{1}}-\ee^{\frac{9-p}{(3-p)^{2}}T_{0}}\right)\leq \Vol\left(\left\lbrace w_p\leq T_1\right\rbrace\smallsetminus \left\lbrace\abs{\nabla w_p}=0\right\rbrace\right)\leq \Vol(B_{R_{T_{1}}}(o))\leq \kst_{\mathrm{vol}}R_{T_{1}}^{1+\alpha}.
\end{equation}
Since $u_p=\ee^{-\frac{w_p}{p-1}}$, by estimate \cref{upper_bound}, we have 
\begin{equation}
   \ee^{-\frac{w_p(x)}{p-1}}\leq \kst\dist(x,o)^{-\frac{\alpha+1-p}{p-1}}.
\end{equation}
Taking the supremum over all $x\in\partial\Omega_{T_{1}}^{\sml{p}}$, we obtain $R_{T_1}^{\alpha+1-p}\leq \kst\ee^{T_1}$, implying $R_{T_1}^{1+\alpha}\leq \kst\ee^{\frac{1+\alpha}{\alpha+1-p}T_1}$, for a positive constant $\kst=\kst(M,\Omega)$. Therefore, by inequality \cref{eq:inequality}, we conclude that
\begin{equation}\label{eq:contradiction_inequality}
 \kst\left(\ee^{\frac{9-p}{(3-p)^2} T_1}- \ee^{\frac{9-p}{(3-p)^2} T_0}\right)\leq \kst_{\mathrm{vol}}R_{T_1}^{1+\alpha}\leq \kst' \ee^{\frac{1+\alpha}{1+\alpha-p}{T_1}},  
\end{equation}
for some $\kst'=\kst'(M,\Omega)$.
This is not possible for $T_1$ sufficiently large, provided $\alpha>4/(5-p)$. Therefore, $\alpha \leq 4/(5-p) $ for every $p \in (1,2)$. Then, $\alpha \leq 1$ which contradicts the superquadratic growth assumption.

\smallskip

Suppose now by contradiction that $(M,g)$ is not flat.  Then there must exist a point $o\in M$ with $\sca(o)>0$ and a radius $r\ll 1$ such that $\partial B_{r}(o)$ is a smooth surface with
\begin{equation}\label{F(0)}
\int_{\partial B_{r}(o)}\H^2\dif\Hff^{2}<16\pi.
\end{equation} 
This can be obtained by the expansion in normal coordinates in perturbed spheres, see, for instance, \cite[Proposition 3.1]{mondino_existencecriticalpointsWillmore_2010} and \cite{fan_largespheresmallspherelimitsbrownyork_2009}. 
Set $\Omega = \overline{B_r(o)}$ in \cref{eq:moser_p_potential} and let $w_p$ be the resulting solution. By the Hopf lemma, zero is a regular value for $w_p$.  In particular,
\begin{equation}
    \MF_p(0) = \int_{\partial B_{r}(o)}\frac{\H^2}{4} - \left(\frac{\H}{2} - \frac{\abs{\nabla w_p}}{3-p}\right)^2\dif\Hff^{2} < 4 \pi.
\end{equation}
$(M,g)$ must be isometric to the Euclidean space since it is the only flat Riemannian $3$-manifold with superquadratic volume growth. This completes the proof in the setting of orientable manifolds. If $(M,g)$ is nonorientable, then its orientable double cover is Ricci-pinched and has superquadratic volume growth. Hence, it must be Euclidean space. This yields a contradiction, since $\R^3$ cannot be the double cover of any nonorientable Riemannian $3$-manifold. \hfill \qed

\begin{remark}[Superquadratic volume growth] Theorem 12 in
    \cite{huisken_inversemeancurvatureflow_2023} is proved under a set of conditions different from \cref{superquadratic}, namely
    \begin{equation}\label{eq:volume_least_growth}
        \liminf_{r \to +\infty} \frac{\abs{B_r(q)}}{r^{1+\alpha}} >0
    \end{equation}
    for some $q \in M$ and 
    \begin{equation}\label{eq:noncollapsed}
        \inf_{p \in M} {\abs{B_1(p)}}>0.
    \end{equation}
    
    It is easy to show that \cref{eq:volume_least_growth} implies the lower bound in \cref{superquadratic}, which is enough to ensure the existence of solution to \cref{eq:p_harmonic_potential} and provide an upper barrier \cref{upper_bound} for it. Beyond this implication, no clear connection seems to exist between the other assumptions.

    \cref{thmpinchingtheorem} can be proved replacing \cref{superquadratic} with \cref{eq:volume_least_growth} and \cref{eq:noncollapsed}. Indeed, these two conditions imply the validity of the isoperimetric inequality \cite[Theorème 3]{coulhon_isoperimetrie_1993}
    \begin{equation}
        \gamma \min\set{\abs{E}^{\frac{2}{3}}, \abs{E}^{\frac{\alpha}{1+\alpha}}} \leq \abs{\partial E},
    \end{equation}
    for some $\gamma >0$.
    By a symmetrization argument (see \cite[Theorem 4.1]{benatti_minkowskiinequalitycompleteriemannian_2024} for example), one can deduce the following isocapacitary inequality
    \begin{equation}\label{eq:isocapacitary}
        \gamma_p \min\set{\abs{E}^{\frac{3-p}{3}}, \abs{E}^{\frac{\alpha+1-p}{1+\alpha}}} \leq \ncapa_p(\partial E),
    \end{equation}
    for some $\gamma_p >0$. Employing \cref{eq:isocapacitary} instead of \cref{superquadratic} in \cref{eq:inequality} one gets \cref{eq:contradiction_inequality} since
    \begin{equation}
        \kst\left(\ee^{\frac{9-p}{(3-p)^{2}}T_{1}}-\ee^{\frac{9-p}{(3-p)^{2}}T_{0}}\right)\leq \gamma_p\Vol(\set{w_p \leq T_1})\leq  \ncapa_p(\set{w_p \leq T_1})^{\frac{1+\alpha}{\alpha +1 -p }}=\kst' \ee^{\frac{1+\alpha}{\alpha +1 -p }T_1},
    \end{equation}
    for a constant $\kst' = \kst'(M, \Omega)$. 

\smallskip

    Huisken and K\"orber also discuss another possible condition under which their theorem can be proved in \cite[Remark 11]{huisken_inversemeancurvatureflow_2023}. That condition implies our \cref{superquadratic}.
\end{remark}

\begin{remark}
Both Deruelle-Simon-Schulze and Lee-Topping considered a generalization to higher dimensions of the Ricci-pinching theorem (requiring the manifold to be PIC1) \cite{deruelle_Hamilton-Lott_2024,lee_manifoldspic1pinchedcurvature_2025}. Their approach is via Ricci flow. It would be interesting to investigate whether the inverse mean curvature flow/nonlinear potential theory could also be applied in this context.
\end{remark}

\subsection{The case of manifolds with boundary}

To the best of the authors’ knowledge, the analogue of \cref{main} for manifolds with boundary has not been systematically investigated so far, except for the recent advances in \cite{benatti_notericcipinchedthreemanifolds_2024}. It is therefore natural to ask whether, also in the presence of a boundary, a Ricci-pinching condition is sufficient to enforce strong geometric rigidity, or whether the boundary allows for greater flexibility. In the same spirit of \cref{thmpinchingtheorem}, we might want to first answer the following question.
\begin{question}\label{question1}
    Let $(M,g)$ be a complete, connected, noncompact, orientable, Ricci-pinched Riemannian $3$-manifold with compact smooth boundary $\partial M \neq \varnothing$. Suppose that $(M,g)$ has superquadratic volume growth. Is $(M,g)$ isometric to the exterior of a compact set in the Euclidean space?
\end{question}

The answer to \cref{question1} is affirmative if one imposes an additional assumption on $\partial M$.

\begin{proposition}\label{thm:boundary}
Let $(M,g)$ be a complete, connected, orientable, noncompact, Ricci-pinched Riemannian $3$-manifold. Suppose that $(M,g)$ has superquadratic volume growth and 
\begin{equation}\label{F(0)bis2}
\int_{\partial M}\H^2\dif\Hff^2\leq 16\pi.
\end{equation}
Then, $(M,g)$ is isometric to the exterior of a ball in the flat Euclidean space and equality holds in \cref{F(0)bis2}.
\end{proposition}

\begin{proof}
    
    Set $\Omega = \partial M$ and $w_p$ be the solution to \cref{eq:moser_p_potential}. If $\MF_p(0)<4\pi$, the analysis provided in the first part of the proof of \cref{main} implies that such a manifold cannot exist. We only have to discuss the case $\MF_p(0)=4\pi$, which gives
    \begin{align}\label{eq:zzstarting}
        \int_{\partial M}\H^2\dif\Hff^2=16\pi && \int_{\partial M}\left(\frac{\H}{2} - \frac{\abs{\nabla w_p}}{3-p }\right)^2\dif\Hff^2=0.
    \end{align}
    In particular, these identities imply $\IF_p'(0)=0$. By \cref{eq:derGin0}, we have that $\MF_p'(t)\equiv 0$ for all $ t \in[0,+\infty)$. The argument is now classical and is based on \cite{huisken_inverse_2001} (see also \cite{benatti_minkowskiinequalitycompleteriemannian_2024}). Let $T$ be the supremum of all $\tau$ for which $w_p$ is smooth without critical point in $\set{0 \leq w_p \leq \tau}$. Then, $T>0$ by Hopf lemma and the continuity of $\abs{\nabla w_p}$. Since $\MF_p'(t)\equiv 0$, one has
    \begin{align}\label{eq:zzzidentitiesgeometric}
        \nabla^\top \abs{\nabla w_p}=0, && \mathring{\h}_t=0, && 2 \abs{\nabla w_p}=(3-p) \H_t, && \Ric(\nu,\nu)=0
    \end{align}
    on $\partial \Omega^{\sml{p}}_t$ for every $t \in [0,T)$, where we used the subscript $t$ to denote the geometric quantities of the respective level set. Hence, $\abs{\nabla w_p}$ (and so is $\H_t$) is positive and constant on $\partial \Omega^{\sml{p}}_t$. By Gauss' lemma applied on level sets of $w_p$, $\set{0\leq w_p< T}$ is diffeomorphic to $\partial M \times [0,T)$ and the metric $g$ can be written as
    \begin{equation}
        g = \frac{\dif t \otimes \dif t }{\abs{\nabla w_p}^2} + g_t =  \frac{4\dif t \otimes \dif t }{(3-p)^2\H_t^2} + g_t.
    \end{equation}
    The evolution equations yield
    \begin{equation}
       \frac{\partial}{\partial t} \H_t = -\Delta_{\partial \Omega^{\sml{p}}_t} \frac{1}{\abs{\nabla w_p} } - \frac{1}{\abs{\nabla w_p}}\left( \abs{\h_t}^2+ \Ric(\nu,\nu)\right)\overset{\cref{eq:zzzidentitiesgeometric}}{=}- \frac{\H_t}{3-p}.
    \end{equation}
    Integrating it, one gets $\H_t= \ee^{-\frac{t}{3-p}} \H_0$. In particular, $\abs{\nabla w_p} \geq \ee^{-\frac{t}{3-p}}\H_0$ on $\set{0 \leq w_p < T}$ hence $T$ cannot be finite. Moreover, integrating the system of partial differential equations
    \begin{equation}
        \frac{\partial}{\partial t} g_{t} =  \frac{2\h_{t} }{\abs{\nabla w_p}}= \frac{\H_t}{\abs{\nabla w_p}}  g_{t}=\frac{2}{3-p} g_{t},
    \end{equation}
    one has $g_{t} = e^{\frac{2t}{3-p}}g_{\partial M}$. Taking $r = \frac{2}{\H_0}\ee^{\frac{t}{3-p}} $ and recalling \cref{eq:zzstarting} we obtain that $(M,g)$ is isometric to 
    \begin{equation}
        \left([0,+\infty) \times \partial M, \dif r \otimes \dif r + \left(\frac{r}{r_0} \right)^2 g_{\partial M}\right),\qquad \text{where } r_0 = \sqrt{\frac{\abs{\partial M}}{4 \pi}}.
    \end{equation}
    Since $\Ric \geq0 $, one can infer $\Ric_{\partial M}\geq (1/r_0^2)g_{\partial M}$. Moreover, $\abs{\partial M } = 4 \pi r_0^2 = \abs{\S^2(r_0)}$. By the rigidity in Bishop--Gromov theorem $(\partial M, g_{\partial M})$ is isometric to $(\S^2, r^2_0 g_{\S^{2}})$ and, in conclusion, $(M,g)$ is isometric to $(\R^3\smallsetminus \set{\abs{x} \leq r_0}, g_{\mathrm{Euc}})$.
\end{proof}

\begin{remark}\label{rmk:nonorientablecone}

    In \cref{thm:boundary}, it is essential that $(M,g)$ is orientable otherwise the contradiction argument used at the end of the proof of \cref{main} does not apply.  The cone over $\mathbb{R}\mathbb{P}^2$ yields a counterexample to the theorem without the orientability assumption. This highlights that all requirements in \cref{question1} are truly essential. 

    On the other hand, one can obtain an \emph{ad hoc} variant of \cref{thm:boundary} within the class of nonorientable manifolds by passing to the orientable double cover. The cone over $\mathbb{R}\mathbb{P}^2$ turns out to be the only flat Ricci-pinched manifold with superquadratic volume growth for which
    \begin{equation}
    \int_{\partial M} \H^2 \dif \Hff^2 \leq 8\pi.
    \end{equation}
    As expected, the presence of the boundary seems to allows for a greater variety of examples, but it is still subject to some limitations.
\end{remark}

\begin{remark} 
For a manifold with boundary, nonnegative Ricci curvature and Euclidean volume growth, one can deduce from \cite{agostiniani_sharpgeometricinequalitiesclosed_2020,wang_inequality} the following Willmore inequality
\begin{equation}\label{eq:willmorelitterature}
\int_{\partial M}\H^2\dif\Hff^2 \geq 16\pi\AVR(g).
\end{equation}
 \cref{thm:boundary} can be rephrased by saying that Ricci-pinched manifold with superquadratic volume growth can only satisfy the Euclidean case of \cref{eq:willmorelitterature}, that is $\AVR(g)=1$. Among the inequalities of the form \cref{eq:willmorelitterature}, the one provided by \cref{thm:boundary} is the strongest and therefore yields a more rigid characterization of the equality case. This is not obvious a priori, since submaximal volume growth typically leads to $\AVR(g)=0$ and the Willmore inequality becomes trivial. In this case the Ricci-pinched condition compensate for the superquadratic volume growth, thus suggesting a positive answer to \cref{question1}. 
\end{remark}

\begingroup
\setlength{\emergencystretch}{1em}
\printbibliography
\endgroup

\end{document}